\documentclass[english]{article}
\usepackage[T1]{fontenc}
\usepackage[latin9]{inputenc}
\usepackage{geometry}
\usepackage{amsmath}
\usepackage{amssymb}
\usepackage{graphicx}

\makeatletter
\newtheorem{theorem}{Theorem} 
\newtheorem{lemma}[theorem]{Lemma} 
\newtheorem{proposition}[theorem]{Proposition}

\newenvironment{proof}[1][Proof]
 {\begin{trivlist} \item[\hskip \labelsep {\bfseries #1}]}{\end{trivlist}}   
\newcommand{\qed}{\nobreak \ifvmode \relax \else       \ifdim\lastskip<1.5em \hskip-\lastskip       \hskip1.5em plus0em minus0.5em \fi \nobreak       \vrule height0.75em width0.5em depth0.25em\fi}

\date{}

\makeatother

\usepackage{babel}
\begin{document}

\title{The root distribution of polynomials with a three-term recurrence}

\author{Khang Tran\\
Department of Statistics\\
Truman State University, USA}
\maketitle
\begin{abstract}
For any fixed positive integer $n$, we study the root distribution
of a sequence of polynomials $H_{m}(z)$ satisfying the rational generating
function
\[
\sum_{m=0}^{\infty}H_{m}(z)t^{m}=\frac{1}{1+B(z)t+A(z)t^{n}}
\]
where $A(z)$ and $B(z)$ are any polynomials in $z$ with complex
coefficients. We show that the roots of $H_{m}(z)$ which satisfy
$A(z)\ne0$ lie on a specific fixed real algebraic curve for all large
$m$. 
\end{abstract}

\section{Introduction}

The sequence of polynomials $H_{m}(z)$, generated by the rational
function $1/(1+B(z)t+A(z)t^{n})$, has the three-term recurrence relation
of degree $n$
\begin{equation}
H_{m}(z)+B(z)H_{m-1}(z)+A(z)H_{m-n}(z)=0\label{eq:recurrence}
\end{equation}
and the initial conditions
\begin{equation}
H_{m}(z)=(-1)^{m}B^{m}(z),\,0\le m<n.\label{eq:initialcond}
\end{equation}
For the study of the root distribution of other sequences of polynomials
that satisfy three-term recurrences, see \cite{cc,hs}. In \cite{tran},
the author shows that in the three special cases when $n=2$, $3$,
and $4$, the roots of $H_{m}(z)$ which satisfies $A(z)\ne0$ will
lie on the curve $\mathcal{C}$ defined in Theorem \ref{maintheorem},
and are dense there as $m\rightarrow\infty$. This paper shows that
for any fixed integer $n$, this result holds for all large $m$ in
the theorem below. 

\begin{theorem}\label{maintheorem}

Let $H_{m}(z)$ be a sequence of polynomials whose generating function
is 
\[
\sum_{m=0}^{\infty}H_{m}(z)t^{m}=\frac{1}{1+B(z)t+A(z)t^{n}}
\]
where $A(z)$ and $B(z)$ are polynomials in $z$ with complex coefficients.
There is a constant $C=C(n)$ such that for all $m>C$, the roots
of $H_{m}(z)$ which satisfy $A(z)\ne0$ lie on a fixed curve $\mathcal{C}$
given by
\[
\Im\frac{B^{n}(z)}{A(z)}=0\qquad\mbox{and}\qquad0\le(-1)^{n}\Re\frac{B^{n}(z)}{A(z)}\le\frac{n^{n}}{(n-1)^{n-1}}
\]
and are dense there as $m\rightarrow\infty$. 

\end{theorem}

This theorem holds when the numerator of the generating function is
a monomial in $t$ and $z$. For a general numerator, it appears,
in an unpublished joint work with Robert Boyer, that the set of roots
will approach $\mathcal{C}$ and a possible finite set in the Hausdorff
metric on the non-empty compact subsets. For more study of sequences
of polynomials whose roots approach fixed curves, see \cite{boyergoh,boyergoh-1}.
Other studies of the limits of zeros of polynomials satisfying a linear
homogeneous recursion whose coefficients are polynomials in $z$ are
given in \cite{bkw,bkw-1}.

An important trinomial is $y^{m}-my+m-1$. Its fundamental role in
the study of inequalities is pointed out in both \cite{bb} and \cite{hlp}.
This paper shows that a ``$\theta$-analogue'' of this ($\theta=0$)
trinomial is fundamental for the study of polynomials generated by
rational functions whose denominators are trinomials. Some further
information about trinomials is avalable in \cite{dns}. It is also
noteworthy that although there is no really concise formula for the
discriminant of a general polynomial in terms of its coefficients,
there is such a formula for the discriminant of a trinomial \cite[pp. 406--407]{gkz}.
Here we develop, in the fashion of Ismail, a $q$-analogue of this
discriminant formula. This plays a fundamental role in the determination
of the curve $\mathcal{C}$. 

Our main approach is to count the number of roots of $H_{m}(z)$ on
the curve $\mathcal{C}$ and show that this number equals the degree
of this polynomial. This number of roots connects with the number
of quotients of roots in $t$ of the denominator $1+B(z)t+A(z)t^{n}$
on a portion of the unit circle. The plot of these quotients when
$n=6$ and $m=30$ is given in Figure 1. Although the curve $\mathcal{C}$
depends on $A(z)$ and $B(z)$, it will be seen that this plot of
the quotients is independent of these two polynomials. 

For an example of Theorem \ref{maintheorem}, we consider the sequence
of Chebyshev polynomials of the second kind, $U_{n}(x)$, where $B(z)=-2z$
and $A(z)=1$. With $z=x+iy$, the curve $\mathcal{C}$ is given by
$xy=0$ and $0\le4(x^{2}-y^{2})\le4$ which is identical to the real
interval $[-1,1]$. In another example when $A(z)=z^{3}+i$, $B(z)=z$,
and $n=5$, the curve $\mathcal{C}$ is a portion of the curve given
by the Cartesian equation
\[
x\left(2x^{6}y+6x^{4}y^{3}-x^{4}+6x^{2}y^{5}+10x^{2}y^{2}+2y^{7}-5y^{4}\right)=0.
\]
The plot of the roots of $H_{m}(z)$ in the latter example is given
in Figure 2. 

\begin{figure}
\begin{centering}
\includegraphics{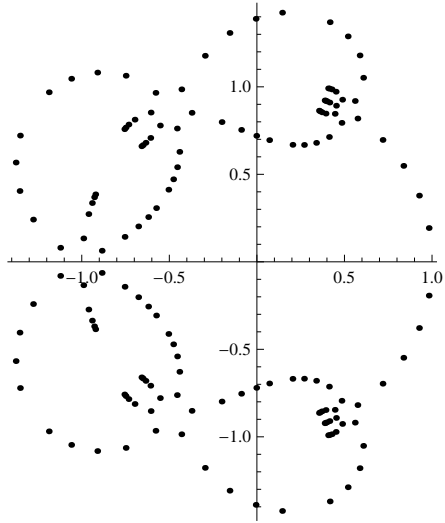}
\par\end{centering}

\protect\caption{Distribution of the quotients of roots of the hexic denominator}
\end{figure}

\begin{figure}
\begin{centering}
\includegraphics{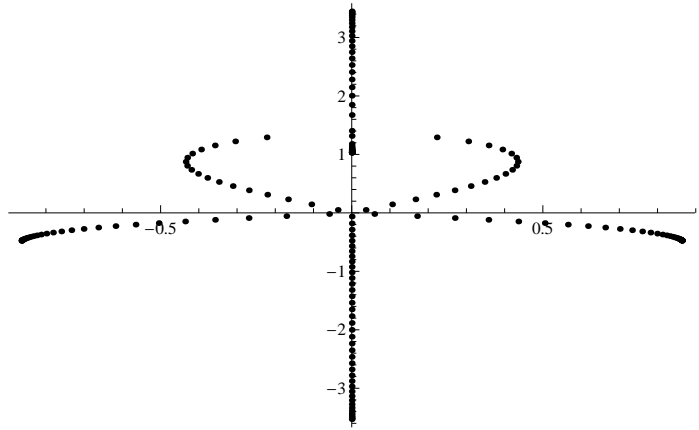}
\par\end{centering}

\protect\caption{Distribution of the roots of $H_{200}(z)$ when $A(z)=z^{3}+i$, $B(z)=z$,
and $n=5$}
\end{figure}

\section{Proof of the theorem}

In this paper, we let $m=np+r$ where $p$ and $r$ are two positive
integers with $0\le r<n$. We also let $a$ and $b$ be the degrees
of the polynomials $A(z)$ and $B(z)$ respectively. 

\begin{lemma}\label{degHz}

The degree of the polynomial $H_{m}(z)$ is at most
\[
\begin{cases}
mb & \mbox{ if }nb>a\\
pa+rb & \mbox{ if }nb\le a
\end{cases}.
\]

\end{lemma}

\begin{proof}

The lemma follows from induction and the definition of $H_{m}(z)$
in (\ref{eq:recurrence}) and (\ref{eq:initialcond}). 

\end{proof}

For each $z$, we let $t_{k}=t_{k}(z)$, $0\le k<n$, be the roots
of the denominator $1+B(z)t+A(z)t^{n}$. We will focus our attention
on the quotients of roots $q_{k}=t_{k}/t_{0}$, $0\le k<n$ because
in Lemma \ref{quotientlemma} below, these quotients satisfy an equation
which is independent of $A(z)$ and $B(z)$. See Figure 1 for the
plot of these quotients when $z$ runs through the roots of $H_{m}(z)$.
This lemma also appears in a different unpublished joint work with
Boyer and a proof is given for completeness.

\begin{lemma}\label{quotientlemma}

Let $z$ be a root of $H_{m}(z)$. If $q:=q_{1}=e^{2i\theta}$, $\theta\in\mathbb{R}$,
then besides the two roots $e^{\pm i\theta}$, the remaining $n-2$
roots of the polynomial
\begin{equation}
P_{\theta}(\zeta)=\zeta^{n}-\frac{\sin n\theta}{\sin\theta}\zeta+\frac{\sin(n-1)\theta}{\sin\theta}\label{eq:Pthetadef}
\end{equation}
are $e^{-i\theta}q_{2},\ldots,e^{-i\theta}q_{n-1}$.

\end{lemma}

\begin{proof}

Let $e_{k}(t_{0},t_{1},\ldots,t_{n-1})$ be the $k$-th elementary
symmetric polynomials in the variables $t_{0},t_{1},\ldots,t_{n-1}$.
Since $t_{0},t_{1},\ldots,t_{n-1}$ are the roots of $1+B(z)t+A(z)t^{n}$,
we have $e_{k}(t_{0},t_{1},\ldots,t_{n-1})=0$ when $1\le k\le n-2$.
Let $e_{k}:=e_{k}(q_{2},\ldots,q_{n-1})$. We divide the equation
$e_{k}(t_{0},t_{1},\ldots,t_{n-1})=0$ by $t_{0}^{k}$, $1\le k\le n-2$,
to obtain
\[
\begin{cases}
e_{1}+q+1 & =0\\
e_{2}+e_{1}(q+1)+q & =0\\
e_{3}+e_{2}(q+1)+e_{1}q & =0\\
e_{4}+e_{3}(q+1)+e_{2}q & =0\\
\vdots & \vdots\\
e_{n-2}+e_{n-3}(q+1)+e_{n-4}q & =0.
\end{cases}
\]
We solve this system of equations for $e_{1},\ldots,e_{n-2}$ by substitution
to get 
\[
e_{k}=(-1)^{k}(1+q+q^{2}+\cdots+q^{k})
\]
with $1\le k\le n-2$. Assuming that $e_{0}=1$, the definition of
$e_{k}$ implies that $q_{2},q_{3},\ldots,q_{n-1}$ are the roots
of the equation 
\[
\sum_{k=0}^{n-2}(-1)^{k}e_{k}\zeta^{n-k-2}=0.
\]
We multiply both side of this equation by $(1-q)$ to obtain 
\[
(1-q)\zeta^{n-2}+(1-q^{2})\zeta^{n-3}+(1-q^{3})\zeta^{n-4}+\cdots+(1-q^{n-1})=0.
\]
Thus 
\[
\zeta^{n-2}+\zeta^{n-3}+\cdots+1=q\zeta^{n-2}+q^{2}\zeta^{n-3}+\cdots+q^{n-1}.
\]
We add $\zeta^{n-1}$ to both sides to get 
\begin{equation}
\frac{\zeta^{n}-1}{\zeta-1}=\frac{\zeta^{n}-q^{n}}{\zeta-q}.\label{eq:zetaquotient}
\end{equation}
This identity implies that 
\[
\zeta^{n}(1-q)-(1-q^{n})\zeta+q(1-q^{n-1})=0.
\]
We divide both sides of this equation by $1-q$ and note that
\begin{eqnarray}
\frac{q^{n}-1}{q-1} & = & e^{(n-1)i\theta}\frac{e^{ni\theta}-e^{-ni\theta}}{e^{i\theta}-e^{-i\theta}}\nonumber \\
 & = & e^{(n-1)i\theta}\frac{\sin n\theta}{\sin\theta}.\label{eq:ChebyshevU}
\end{eqnarray}
We obtain 
\[
\zeta^{n}-e^{(n-1)i\theta}\frac{\sin n\theta}{\sin\theta}\zeta+e^{ni\theta}\frac{\sin(n-1)\theta}{\sin\theta}=0.
\]
The lemma follows from a substitution of $\zeta$ by $e^{i\theta}\zeta$. 

\end{proof}

The quotients of the roots of a polynomial connect with the $q$-analogue
of its discriminant. The $q$-discriminant, introduced by Mourad Ismail\cite{ismail},
of a general polynomial $P(x)$ of degree $d_{P}$ and the leading
coefficient $a_{P}$ is 
\[
\mathrm{Disc}_{x}(P;q)=a_{P}^{2n-2}q^{n(n-1)/2}\prod_{1\le i<j\le d_{P}}(q^{-1/2}x_{i}-q^{1/2}x_{j})(q^{1/2}x_{i}-q^{-1/2}x_{j}).
\]
where $x_{i}$, $1\le i\le d_{P}$, are the roots of $P(x)$. This
$q$-discriminant is zero if and only if a quotient of roots $x_{i}/x_{j}$
equals $q$. When $q$ approaches $1$, this $q$-discriminant becomes
the ordinary discriminant $\mathrm{Disc}_{x}P(x)$. This ordinary
discriminant is also the resultant of $P(x)$ and its derivative.
The resultant of any two polynomials $P(x)$ and $R(x)$ is given
by 
\begin{eqnarray}
\mathrm{Res}_{x}(P(x),R(x)) & = & a_{P}^{d_{R}}\prod_{P(x_{i})=0}R(x_{i})\nonumber \\
 & = & a_{R}^{d_{P}}\prod_{R(x_{i})=0}P(x_{i})\label{eq:resdef}
\end{eqnarray}
where $a_{R}$ and $d_{R}$ are the leading coefficient and the degree
of $R(x)$ respectively. For further information on the ordinary discriminant
and resultant, see \cite{aar,apostal,dilcherstolarsky,gkz,gisheismail}.

Ismail \cite{ismail} showed that

\begin{proposition}\label{qdiscform}

The $q$-discriminant of a polynomial $P(x)$ of degree $n$ with
a lead coefficient $\gamma$ is 
\[
\mathrm{Disc}_{x}(P;q)=(-1)^{n(n-1)/2}\gamma^{n-2}\prod_{i=1}^{n}(D_{q}P)(x_{i})
\]
where 
\[
(D_{q}P)(x)=\frac{P(x)-P(qx)}{x-qx}.
\]

\end{proposition}

The lemma below also appears in an unpublished joint work with Boyer.
A proof is given for completeness.

\begin{lemma}\label{qdiscdenom}

The $q$-discriminant of $D(t)=1+Bt+At^{n}$ is 
\[
\mathrm{Disc}_{t}(D(t);q)=\pm A^{n-2}\left(B^{n}\frac{q^{n-1}(1-q^{n-1})^{n-1}}{(1-q)^{n-1}}+(-1)^{n-1}\frac{(1-q^{n})^{n}}{(1-q)^{n}}A\right).
\]

\end{lemma}

\begin{proof}

Proposition \ref{qdiscform} gives 

\begin{eqnarray*}
\mathrm{Disc}_{t}(D(t);q) & = & \pm A^{n-2}\prod_{D(t)=0}(D_{q}D)(t)
\end{eqnarray*}
where 
\begin{eqnarray}
(D_{q}D)(t) & = & \frac{D(t)-D(qt)}{t-qt}\nonumber \\
 & = & B+At^{n-1}\frac{1-q^{n}}{1-q}.\label{eq:qdiffdenom}
\end{eqnarray}
Using the symmetric definition of resultant in (\ref{eq:resdef}),
we write the product in $\mathrm{Disc}_{t}(D(t);q)$ as 
\begin{equation}
\mathrm{Disc}_{t}(D(t);q)=\pm A^{n-1}\left(\frac{1-q^{n}}{1-q}\right)^{n}\prod_{(D_{q}D)(t)=0}D(t).\label{eq:qdiscprod}
\end{equation}
From the definition of $D(t)$ and the formula of $(D_{q}D)(t)$ in
(\ref{eq:qdiffdenom}), these two polynomials are connected by 
\[
\frac{1-q^{n}}{1-q}D(t)=t(D_{q}D)(t)+tB\left(\frac{1-q^{n}}{1-q}-1\right)+\frac{1-q^{n}}{1-q}.
\]
Thus (\ref{eq:qdiscprod}) gives
\begin{eqnarray*}
\mathrm{Disc}_{t}(D(t);q) & = & \pm A^{n-1}\left(\frac{1-q^{n}}{1-q}\right)\prod_{(D_{q}D)(t)=0}\left(tB\left(\frac{1-q^{n}}{1-q}-1\right)+\frac{1-q^{n}}{1-q}\right)\\
 & = & \pm A^{n-2}B^{n}\frac{q^{n}(1-q^{n-1})^{n}}{(1-q)^{n}}(D_{q}D)\left(-\frac{1-q^{n}}{Bq(1-q^{n-1})}\right)\\
 & = & \pm A^{n-2}\left(B^{n}\frac{q^{n-1}(1-q^{n-1})^{n-1}}{(1-q)^{n-1}}+(-1)^{n-1}\frac{(1-q^{n})^{n}}{(1-q)^{n}}A\right).
\end{eqnarray*}

\end{proof}

From Lemma \ref{qdiscform}, we see that in the case $q=1$ all the
points $z$ such that $\mathrm{Disc}_{t}D(t)=0$ belong to the curve
$\mathcal{C}$. Thus we only need to consider the case $\mathrm{Disc}_{t}D(t)\ne0$,
i.e., all the roots $t_{0},t_{1},\ldots,t_{n-1}$ are distinct. The
polynomial $P_{\theta}(\zeta)$ in Lemma \ref{quotientlemma} plays
an important role in the proof of Theorem \ref{maintheorem}. In particular,
instead of counting the number of roots of $H_{m}(z)$ on $\mathcal{C}$,
we will count the number of real roots of a function of $\theta$
where $q_{1}=e^{2i\theta}$. This function is given by the lemma below.

\begin{lemma}\label{HthetaLemma}

Let $z$ be a point on $\mathcal{C}$ such that $q_{1}=e^{2i\theta}$,
$\theta\in\mathbb{R}$ and $A(z)\ne0$. Then $z$ is a root of $H_{m}(z)$
if and only if 
\[
h(\theta):=\sum_{k=0}^{n-1}\frac{1}{\zeta_{k}^{m+1}P'_{\theta}(\zeta_{k})}=0
\]
 where $\zeta_{0},\ldots,\zeta_{n-1}$ are the roots of
\[
P_{\theta}(\zeta)=\zeta^{n}-\frac{\sin n\theta}{\sin\theta}\zeta+\frac{\sin(n-1)\theta}{\sin\theta}.
\]

\end{lemma}

\begin{proof}

By partial fractions, we have 
\begin{eqnarray*}
\sum_{m=0}^{\infty}H_{m}(z)t^{m} & = & \frac{1}{A(z)(t-t_{0})(t-t_{2})\cdots(t-t_{n-1})}\\
 & = & \frac{1}{A(z)}\sum_{k=0}^{n-1}\frac{1}{t-t_{k}}\prod_{l\ne k}\frac{1}{t_{k}-t_{l}}\\
 & = & \frac{1}{A(z)}\sum_{k=0}^{n-1}\frac{1}{t_{k}^{m+1}}\prod_{l\ne k}\frac{1}{t_{k}-t_{l}}t^{m}.
\end{eqnarray*}
Thus $H_{m}(z)=0$ is equivalent to 
\[
\sum_{k=0}^{n-1}\frac{1}{t_{k}^{m+1}}\prod_{l\ne k}\frac{1}{t_{k}-t_{l}}=0.
\]
We multiply the equation by $t_{0}^{m+n}$ and let $q_{k}=t_{k}/t_{0}$
to obtain 
\begin{equation}
\sum_{k=0}^{n-1}\frac{1}{q_{k}^{m+1}}\prod_{l\ne k}\frac{1}{q_{k}-q_{l}}=0.\label{eq:Hquotients}
\end{equation}
Let $\zeta_{k}=e^{-i\theta}q_{k}$, $0\le k<n$. Lemma \ref{quotientlemma}
implies that $\zeta_{0},\ldots,\zeta_{n-1}$ are the roots of
\[
P_{\theta}(\zeta)=\zeta^{n}-\frac{\sin n\theta}{\sin\theta}\zeta+\frac{\sin(n-1)\theta}{\sin\theta}
\]
where $\zeta_{0}=e^{-i\theta}$ and $\zeta_{1}=e^{i\theta}$. We multiply
(\ref{eq:Hquotients}) by $e^{(m+n)i\theta}$ and obtain
\[
\sum_{k=0}^{n-1}\frac{1}{\zeta_{k}^{m+1}P'_{\theta}(\zeta_{k})}=0.
\]

\end{proof}

Since $P_{\theta}(\zeta)$ is a real polynomial in $\zeta$, the function
$h(\theta)$ is a real-valued function of $\theta$ by the symmetric
reduction. Each real root $\theta$ of $h(\theta)$ yields some roots
$z$ of $H_{m}(z)$ on $\mathcal{C}.$ For example, if $\theta=\pi/n$
is a root of $h(\theta)$ then a root of $B(z)$ will also be a root
of $H_{m}(z)$ on $\mathcal{C}$ since, from Lemma \ref{qdiscdenom},
\begin{equation}
(-1)^{n}\frac{B^{n}(z)}{A(z)}=\frac{(1-q^{n})^{n}}{(1-q)q^{n-1}(1-q^{n-1})^{n-1}}\label{eq:zqtheta}
\end{equation}
where $q=e^{2i\theta}$. The lemma below gives the number of roots
of $H_{m}(z)$ in the case of $\theta=\pi/n$. 

\begin{lemma}\label{zerosB}

If $m=np+r$ with $0\le r<n$ then the polynomial $B^{r}(z)$ divides
$H_{m}(z)$. 

\end{lemma}

\begin{proof}

The generating function of $H_{m}(z)$ is 
\begin{eqnarray*}
\frac{1}{1+B(z)t+A(z)t^{n}} & = & \sum_{k=0}^{\infty}(-1)^{k}\left(B(z)t+A(z)t^{n}\right)^{k}\\
 & = & \sum_{k=0}^{\infty}\sum_{i=0}^{k}(-1)^{k}\binom{k}{i}t^{n(k-i)+i}A^{k-i}(z)B^{i}(z).
\end{eqnarray*}
The lemma follows from the power of $t$ in the double summation above.

\end{proof}

Our main goal is to show that the number of real roots $\theta$ of
$h(\theta)$ on the interval $[0,\pi/n)$ is $p$ when $m$ is large.
See Figure 1 for a plot of these $p$ roots on the arc $q=e^{2i\theta}$.
Assuming this fact, which will be proved later, we provide an argument
which shows that the roots of $H_{m}(z)$ lie on $\mathcal{C}$. From
Lemma \ref{qdiscdenom}, if $q$ is a quotient of two roots in $t$
of $D(t,z)$ then
\begin{equation}
(-1)^{n}B^{n}(z)(1-q)q^{n-1}(1-q^{n-1})^{n-1}=(1-q^{n})^{n}A(z).\label{eq:zqconnection}
\end{equation}
If $\theta$ is a root of $h(\theta)$ on the interval $[0,\pi/n)$
then, with $q=e^{2i\theta}$, each solution in $z$ of (\ref{eq:zqconnection})
is a root of $H_{m}(z)$. From (\ref{eq:ChebyshevU}), we write the
right side of (\ref{eq:zqtheta}) in $\theta$ and see that it satisfies
\begin{equation}
0\le\frac{\sin^{n}n\theta}{\sin\theta\sin^{n-1}(n-1)\theta}\le\frac{n^{n}}{(n-1)^{n-1}}\label{eq:thetaineq}
\end{equation}
since $\sin n\theta/\sin\theta\le n$ and $\sin n\theta/\sin(n-1)\theta\le n/(n-1)$.
Thus each solution $z$ of (\ref{eq:zqconnection}) belongs to $\mathcal{C}$.
In the case $nb>a$, the equation (\ref{eq:zqconnection}) shows that
each value of $q$ gives $nb$ solutions of $z$ counting multiplicities.
Thus we have at least $nbp$ solutions of $H_{m}(z)$ on $\mathcal{C}$.
Lemma \ref{zerosB} gives us at least $rb$ more solutions of $H_{m}(z)$
on $\mathcal{C}$. Thus the number of roots of $H_{m}(z)$ on $\mathcal{C}$
is at least $nbp+rb=mb$. With a similar argument, when $nb\le a$,
the number of roots of $H_{m}(z)$ on $\mathcal{C}$ is at least $pa+rb$.
This number of roots equals the degree of $H_{m}(z)$ in Lemma \ref{degHz}.
Hence all the roots of $H_{m}(z)$ lie on $\mathcal{C}$. 

Let $\zeta_{0},\ldots,\zeta_{n-1}$ be the roots of $P_{\theta}(\zeta)$
for $\theta\in[0,\pi/n).$ We will show that the number of real roots
of
\begin{equation}
h(\theta)=\sum_{k=0}^{n-1}\frac{1}{\zeta_{k}^{m+1}P'_{\theta}(\zeta_{k})},\label{eq:Htheta}
\end{equation}
is $p$ where $m=np+r$, $0\le r<n$. With the convention that $\sin n\theta/\sin\theta=n$
when $\theta=0$, we have the lemma below.

\begin{lemma}\label{rootsP}

If $0\le\theta<\pi/n$ then, besides the two roots $e^{\pm i\theta},$
the remaining $n-2$ roots of the polynomial
\[
P_{\theta}(\zeta)=\zeta^{n}-\frac{\sin n\theta}{\sin\theta}\zeta+\frac{\sin(n-1)\theta}{\sin\theta}
\]
lie outside the closed unit disk, i.e., $|\zeta|>1$. 

\end{lemma}

\begin{proof}

We replace $\zeta$ by $\zeta e^{-i\theta}$ in $P_{\theta}(\zeta)$
and then rewrite (\ref{eq:zetaquotient}) as 
\[
\frac{\zeta^{n}-1}{\zeta-1}=\frac{q^{n}-1}{q-1}.
\]
From (\ref{eq:ChebyshevU}), the polynomial map
\[
f(z)=\frac{z^{n}-1}{z-1}
\]
bijectively maps the arc $q=e^{2i\theta}$, $-\pi/n<\theta<\pi/n$,
to the biggest loop such as the one in in Figure 3 when $n=6$. Thus
by the Argument Principle Theorem it maps the open unit disk into
the interior of this loop. Hence the equation $f(\zeta)=f(q)$ does
not have a solution $|\zeta|\le1$ except $\zeta=q$. 

\begin{figure}
\begin{centering}
\includegraphics{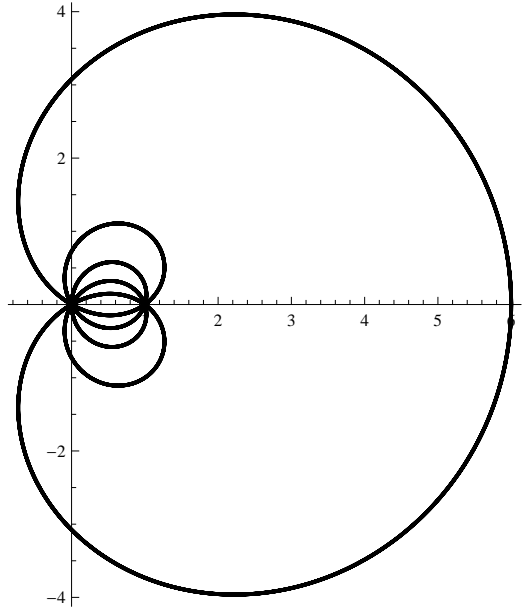}
\par\end{centering}

\protect\caption{The map $h(z)$ on the unit circle when $n=6$}
\end{figure}

\end{proof}

We will find $p+1$ values for $\theta$ where the sign of $h(\theta)$
alternates when $m$ is large and then apply the Intermediate Value
Theorem to obtain $p$ real roots of $h(\theta)$. Lemma \ref{rootsP}
and the definition of $h(\theta)$ in (\ref{eq:Htheta}) imply that
if $\theta$ does not approach $\pi/n$ as $m\rightarrow\infty$ and
$P'(\zeta_{k})\ne0$, the sign of $h(\theta)$ depends on the two
summands when $k=0$ and $k=1$ when $m$ is large. The second condition,
$P'(\zeta_{k})\ne0$, is given below.

\begin{lemma}\label{rootsdistinct}

If $0<\theta<\pi/n$ then the roots of the polynomial
\[
P_{\theta}(\zeta)=\zeta^{n}-\frac{\sin n\theta}{\sin\theta}\zeta+\frac{\sin(n-1)\theta}{\sin\theta}
\]
are distinct. When $\theta=0$, besides the double root at $\zeta=1$,
the remaining $n-2$ roots are distinct.

\end{lemma}

\begin{proof}

If $\zeta$ is a root of $P_{\theta}(\zeta)$ then 
\begin{eqnarray*}
\zeta P'_{\theta}(\zeta) & = & n\zeta^{n}-\frac{\sin n\theta}{\sin\theta}\zeta\\
 & = & (n-1)\frac{\sin n\theta}{\sin\theta}\zeta-n\frac{\sin(n-1)\theta}{\sin\theta}.
\end{eqnarray*}
By the symmetry in the product definition of $\mathrm{Res}_{\zeta}(P_{\theta}(\zeta),P'_{\theta}(\zeta))$
in (\ref{eq:resdef}), the roots of $P_{\theta}(\zeta)$ are distinct
if and only if 
\[
P_{\theta}\left(\frac{n\sin(n-1)\theta}{(n-1)\sin n\theta}\right)\ne0.
\]
With the definition of $P_{\theta}(\zeta)$, this condition becomes
\[
\frac{\sin^{n}n\theta}{\sin\theta\sin^{n-1}(n-1)\theta}\ne\frac{n^{n}}{(n-1)^{n-1}}.
\]
From (\ref{eq:thetaineq}), the equation occurs only when $\theta=0$.
In this case, besides the double root at $1$, the remaining roots
of $P_{0}(\zeta)=\zeta^{n}-n\zeta+(n-1)\zeta$ are distinct. 

\end{proof}

We recall that $m=np+r$, $0\le r<n$. The lemma below shows how $h(\theta)$
alternates sign for $\theta$ not approaching $\pi/n$ as $m\rightarrow\infty$. 

\begin{lemma}\label{range3}

Let $\gamma$ be a fixed small real number and 
\[
\theta=\frac{\pi}{n}-\frac{l\pi}{m},\mbox{\qquad\ensuremath{\mbox{where }\qquad}}l=h+\frac{r}{n}
\]
and $h$ is a non-negative integer. If 
\[
\gamma m<l<\frac{m}{n},
\]
then the sign of $h(\theta)$ is $(-1)^{p-h+1}$ when $m$ is large. 

\end{lemma}

\begin{proof}

If $2\le k\le n-1$ then Lemmas \ref{rootsP} and \ref{rootsdistinct}
imply that $|\zeta_{k}|>1+\epsilon$ and $|P'_{\theta}(\zeta_{k})|>\epsilon$
for some small $\epsilon>0$ which does not depend on $m$. Thus 
\[
\frac{1}{\zeta_{k}^{m+1}P'_{\theta}(\zeta_{k})}
\]
approaches $0$ exponentially when $m$ approaches $\infty$. 

Since $\zeta_{0}=e^{-i\theta}$ and $\zeta_{1}=e^{i\theta}$, the
sum of the first two terms of (\ref{eq:Htheta}) when $k=0$ and $k=1$
is
\[
\frac{2\Re\left(e^{i(m+1)\theta}P'_{\theta}(e^{i\theta})\right)}{|P'(e^{i\theta})|^{2}}.
\]
With fact that $P'_{\theta}(\zeta)=n\zeta^{n-1}-\sin n\theta/\sin\theta$,
this sum becomes
\begin{equation}
\frac{2n\cos(m+n)\theta-2\cos(m+1)\theta\sin n\theta/\sin\theta}{|P'_{\theta}(e^{i\theta})|^{2}}.\label{eq:summainterms}
\end{equation}
When $\theta=\pi/n-l\pi/m=(p-h)\pi/m$ we have $\cos m\theta=(-1)^{p-h}$.
The numerator of (\ref{eq:summainterms}) becomes
\begin{equation}
(-1)^{p-h+1}\frac{2}{\sin\theta}\left(\cos\theta\sin n\theta-n\cos n\theta\sin\theta\right).\label{eq:mainterms-h}
\end{equation}
Since $\theta=(p-h)\pi/m>0$ where $p$ and $h$ are non-negative
integers, we have $\theta\ge\pi/m>1/m$. The function $\cos\theta\sin n\theta-n\cos n\theta\sin\theta$
is positive and increasing on $[0,\pi/n]$ since its derivative is
$(n^{2}-1)\sin\theta\sin n\theta\ge0$. Thus the condition $\theta>1/m$
implies 
\[
\cos\theta\sin n\theta-n\cos n\theta\sin\theta>\cos\frac{1}{m}\sin\frac{n}{m}-n\cos\frac{n}{m}\sin\frac{1}{m}.
\]
When $m$ is large, the right side is close to $n^{3}/2m^{3}$ which
is larger than all of the exponentially decaying terms where $2\le k<n$.
Thus the sign of $h(\theta)$ is $(-1)^{p-h+1}$ by (\ref{eq:mainterms-h}). 

\end{proof}

We next consider the case when $\theta$ approaches $\pi/n$. 

\begin{lemma}

Let $\theta=\pi/n-l\pi/m$ where $l\in\mathbb{R}^{+}$. There is a
sufficiently small $\delta$ such that if $l<\delta m$ then 
\begin{eqnarray}
P_{\theta}(\zeta) & = & \zeta^{n}+1-\frac{nl\pi}{m\sin(\pi/n)}\zeta+\frac{\cos(\pi/n)}{\sin(\pi/n)}.\frac{nl\pi}{m}+O_{n}\left(\frac{l^{2}}{m^{2}}\right),\label{eq:Pestimate}\\
P'_{\theta}(\zeta) & = & n\zeta^{n-1}-\frac{\sin(nl\pi/m)}{\sin(\pi/n)}+O_{n}\left(\frac{l^{2}}{m^{2}}\right),\label{eq:derivPestimate}\\
\zeta_{k} & = & e_{k}\left(1+\frac{l\pi(\cos(\pi/n)-e_{k})}{m\sin(\pi/n)}\right)+O_{n}\left(\frac{l^{2}}{m^{2}}\right).\label{eq:rootestimate}
\end{eqnarray}

\end{lemma}

\begin{proof}

The identity (\ref{eq:Pestimate}) follows from the definition of
$P_{\theta}(\zeta)$ in (\ref{eq:Pthetadef}) and basic asymptotic
computations. We differentiate $P_{\theta}(\zeta)$ and get
\begin{eqnarray*}
P'_{\theta}(\zeta) & = & n\zeta^{n-1}-\frac{\sin n\theta}{\sin\theta}\\
 & = & n\zeta^{n-1}-\frac{\sin(nl\pi/m)}{\sin(\pi/n)}+O_{n}\left(\frac{l^{2}}{m^{2}}\right).
\end{eqnarray*}
Suppose $\zeta_{k}=e_{k}+\epsilon$ where $\epsilon\in\mathbb{C}$
and $e_{k}=e^{(2k-1)i\pi/n}$, $0\le k\le n-1$. The facts that $P_{\theta}(\zeta_{k})=0$
and (\ref{eq:Pestimate}) imply that 
\[
-\frac{n\epsilon}{e_{k}}+\frac{nl\pi}{m\sin\pi/n}\left(\cos\frac{\pi}{n}-e_{k}\right)+O_{n}\left(\epsilon^{2}+\frac{l\epsilon}{m}+\frac{l^{2}}{m^{2}}\right)=0.
\]
This gives
\begin{eqnarray*}
\epsilon & = & \frac{l\pi e_{k}}{m\sin\pi/n}\left(\cos\frac{\pi}{n}-e_{k}\right)+O_{n}\left(\epsilon^{2}+\frac{l\epsilon}{m}+\frac{l^{2}}{m^{2}}\right)\\
 & = & \frac{l\pi e_{k}}{m\sin\pi/n}\left(\cos\frac{\pi}{n}-e_{k}\right)+O_{n}\left(\frac{l^{2}}{m^{2}}\right).
\end{eqnarray*}
Thus 
\[
\zeta_{k}=e_{k}\left(1+\frac{l\pi(\cos\pi/n-e_{k})}{m\sin\pi/n}\right)+O_{n}\left(\frac{l^{2}}{m^{2}}\right).
\]

\end{proof}

\begin{lemma}\label{range2}

Let $\delta>0$ be a fixed real number and 
\[
\theta=\frac{\pi}{n}-\frac{l\pi}{m},\qquad\mbox{where}\qquad l=h+\frac{r}{n}
\]
and $h$ is a non-negative integer. There is a sufficiently small
$\gamma$ such that if $\delta\sqrt{m}<l<\gamma m$ the sign of 
\[
\sum_{k=0}^{n-1}\frac{1}{\zeta_{k}^{m+1}P'_{\theta}(\zeta_{k})}
\]
is $(-1)^{p-h+1}$ when $m$ is large. 

\end{lemma}

\begin{proof}

From (\ref{eq:rootestimate}), we can choose $\gamma$ small enough
so that 
\begin{eqnarray*}
|\zeta_{k}| & > & 1+\frac{l\pi(\cos\pi/n-\cos(2k-1)\pi/n)}{2m\sin\pi/n}\\
 & > & 1+\frac{\delta\pi(\cos\pi/n-\cos(2k-1)\pi/n)}{2\sqrt{m}\sin\pi/n}.
\end{eqnarray*}
Thus if $k\ne0,1$ then the inequality above and (\ref{eq:derivPatroot})
imply that 
\[
\frac{1}{\left|\zeta_{k}^{m+1}P'(\zeta_{k})\right|}
\]
approaches $0$ when $m\rightarrow\infty$. From (\ref{eq:summainterms}),
the sum of two terms when $k=0,1$ is
\[
\frac{2n\cos(m+n)\theta-2\cos(m+1)\theta\sin n\theta/\sin\theta}{|P'_{\theta}(e^{i\theta})|^{2}}.
\]
When $\gamma$ is small, the quantity $\cos(m+1)\theta\sin n\theta/\sin\theta$
is small since 
\[
\theta=\frac{\pi}{n}-\frac{l\pi}{m}.
\]
Also $\left|P'_{\theta}(e^{i\theta})\right|^{2}$ is close to $n^{2}$
by (\ref{eq:derivPatroot}). Thus the sign of (\ref{eq:Htheta}) is
determined by the sign of 
\[
\cos\left(m\theta+n\theta\right)=(-1)^{p-h}\cos n\theta,
\]
which is $(-1)^{p-h+1}$, when $m$ is large.

\end{proof}

\begin{lemma}\label{range1}

Let 
\[
\theta=\frac{\pi}{n}-\frac{l\pi}{m},\qquad\mbox{where}\qquad l=h+\frac{r}{n}
\]
and $h$ is a non-negative integer. There is a sufficiently small
$\delta$ so that if $1\le l<\delta\sqrt{m}$ the sign of 
\[
h(\theta)=\sum_{k=0}^{n-1}\frac{1}{\zeta_{k}^{m+1}P'_{\theta}(\zeta_{k})}
\]
is $(-1)^{p-h+1}$ when $m$ is large. 

\end{lemma}

\begin{proof}

With a sufficiently small $\delta$, the equation (\ref{eq:rootestimate})
gives
\begin{eqnarray*}
\zeta_{k}^{m+1} & = & e_{k}^{m+1}\left(1+\frac{l\pi(\cos\pi/n-e_{k})}{m\sin\pi/n}\right)^{m}\left(1+O_{n}\left(\frac{l^{2}}{m}\right)\right)\\
 & = & e_{k}^{m+1}e^{l\pi(\cos\pi/n-e_{k})/\sin\pi/n}\left(1+O_{n}\left(\frac{l^{2}}{m}\right)\right).
\end{eqnarray*}
Also equations (\ref{eq:derivPestimate}) and (\ref{eq:rootestimate})
imply that 
\begin{eqnarray}
P'_{\theta}(\zeta_{k}) & = & ne_{k}^{n-1}\left(1+O_{n}\left(\frac{l}{m}\right)\right).\label{eq:derivPatroot}
\end{eqnarray}
We combine (\ref{eq:derivPatroot}) with the estimation of $\zeta_{k}^{m+1}$
to get 
\begin{equation}
\zeta_{k}^{m+1}P'_{\theta}(\zeta_{k})=ne_{k}^{m+n}e^{l\pi(\cos\pi/n-e_{k})/\sin\pi/n}\left(1+O_{n}\left(\frac{l^{2}}{m}\right)\right).\label{eq:bigthetasummand}
\end{equation}
According to (\ref{eq:bigthetasummand}), the summation of the two
terms of $h(\theta)$ when $k=0$ and $k=1$ is

\[
\frac{2}{n}\Re\left(e_{k}^{m+n}e^{-l\pi i}\right)=\frac{2}{n}\Re e^{\pi i(m-ln+n)/n}\left(1+O_{n}\left(\frac{l^{2}}{m}\right)\right).
\]
With $l=h+r/n$, the expression becomes
\begin{equation}
\frac{(-1)^{p-h+1}2}{n}\left(1+O_{n}\left(\frac{l^{2}}{m}\right)\right).\label{eq:maintermsmalltheta}
\end{equation}
We now consider the remaining terms of $h(\theta)$. From (\ref{eq:bigthetasummand}),
the summation of these terms is bounded by
\[
\frac{1}{n}\sum_{k=2}^{n-1}\frac{1}{e^{\pi\left(\cos\pi/n-\cos(2k-1)\pi/n\right)/\sin\pi/n}}\left(1+O_{n}\left(\frac{l^{2}}{m}\right)\right).
\]
We use computer algebra to check that if $n<90$ then 
\[
\frac{1}{n}\sum_{k=2}^{n-1}\frac{1}{e^{\pi(\cos\pi/n-\cos(2k-1)\pi/n)/\sin\pi/n}}<\frac{2}{n}.
\]
From (\ref{eq:maintermsmalltheta}), the lemma holds when $n<90$. 

We now consider the case $n\ge90$. From (\ref{eq:bigthetasummand}),
the sign of $h(\theta)$ is the same as that of 
\[
\frac{1}{n}\sum_{k=0}^{n-1}e_{k}^{-m-n}e^{l\pi e_{k}/\sin\pi/n}.
\]
We can write this summation as 
\[
\frac{1}{n}\sum_{k=0}^{n-1}e_{k}^{-m-n}e^{lne_{k}}+\epsilon
\]
 where $\epsilon$ is a small number. The Taylor expansion gives 
\[
\frac{1}{n}\sum_{k=0}^{n-1}e_{k}^{-m-n}e^{lne_{k}}=\frac{1}{n}\sum_{k=0}^{n-1}e_{k}^{-m-n}\sum_{j=0}^{\infty}\frac{(lne_{k})^{j}}{j!}.
\]
With the fact that 
\[
\sum_{k}e_{k}^{a}=\begin{cases}
0 & \mbox{ if }n\nmid a\\
n(-1)^{a/n} & \mbox{ if }n|a
\end{cases},
\]
the double summation becomes
\[
\sum_{j=0}^{\infty}(-1)^{j+p+1}\frac{(ln){}^{jn+q}}{(jn+q)!}.
\]
With $l=h+q/n$, this summation is 
\begin{equation}
\sum_{j=0}^{\infty}(-1)^{j+p+1}\frac{(nh+q)^{nj+q}}{(nj+q)!}.\label{eq:altseries}
\end{equation}
By taking the natural logarithm, we leave it to the reader to check
that the absolute value of the summand increases from $j=0$ to $j=h$
and then decreases when $j>h$. Thus the alternating signs imply that
\[
\left|\sum_{j=0}^{h-1}(-1)^{j+p+1}\frac{(nh+q)^{nj+q}}{(nj+q)!}\right|\le\frac{(nh+q)^{n(h-1)+q}}{(n(h-1)+q)!},
\]
and 
\[
\left|\sum_{j=h+1}^{\infty}(-1)^{j+p+1}\frac{(nh+q)^{nj+q}}{(nj+q)!}\right|\le\frac{(nh+q)^{n(h+1)+q}}{(n(h+1)+1)!}.
\]
Since 
\[
\frac{(nh+q)^{n(h-1)+q}}{(n(h-1)+q)!}+\frac{(nh+q)^{n(h+1)+q}}{(n(h+1)+1)!}<\frac{(nh+q)^{nh+q}}{(nh+q)!},
\]
the sign of (\ref{eq:altseries}) is $(-1)^{h+p+1}$. The lemma follows.

\end{proof}

\begin{lemma}\label{alternatesigns}

Let 
\[
\theta=\frac{\pi}{n}-\frac{l\pi}{m},\qquad\mbox{where}\qquad l=h+\left\{ \frac{m}{n}\right\} ,
\]
$h$ is an non-negative integer, and $l\ge1$. The sign of 
\[
\sum_{k=0}^{n-1}\frac{1}{\zeta_{k}^{m+1}P'_{\theta}(\zeta_{k})}
\]
is $(-1)^{\left\lfloor m/n\right\rfloor -h+1}$ when $m$ is large. 

\end{lemma}

\begin{proof}

This lemma follows from Lemmas (\ref{range1}), (\ref{range2}), and
(\ref{range3}).

\end{proof}

\begin{lemma}\label{smalltheta}

The function $h(\theta)$ satisfies $h(0^{+})<0$. 

\end{lemma}

\begin{proof}

\end{proof}

Equation (\ref{eq:summainterms}) in the proof of Lemma \ref{range3}
shows that the main term of $h(\theta)$ is 
\[
\frac{2n\cos(m+n)\theta-2\cos(m+1)\theta\sin n\theta/\sin\theta}{|P'_{\theta}(e^{i\theta})|^{2}}.
\]
The numerator of this is 
\begin{eqnarray*}
 &  & 2n\left(1-\frac{(m+n)^{2}\theta^{2}}{2}+O(m^{4}\theta^{4})\right)\\
 &  & -2n\left(1-\frac{(m+1)^{2}\theta^{2}}{2}+O(m^{4}\theta^{4})\right)(1+O(\theta^{2})).
\end{eqnarray*}
The claim follows.

\begin{lemma}\label{largetheta}

The sign of 
\[
h\left(\frac{\pi}{n}^{-}\right)
\]
is $(-1)^{p+1}$.

\end{lemma}

\begin{proof}

Let $m=pn+q$, $0\le q<n$. We consider equation (\ref{eq:bigthetasummand})
\begin{eqnarray*}
\zeta_{k}^{m+1}P_{\theta}'(\zeta_{k}) & = & ne_{k}^{m+n}e^{l\pi(\cos\pi/n-e_{k})/\sin\pi/n}\left(1+O_{n}\left(\frac{l^{2}}{m}\right)\right)
\end{eqnarray*}
where $\theta=\pi/n-l\pi/m$. By using the first $n$ terms of the
Taylor expansion for $l$ sufficiently small, we see that the sign
of $h(\theta)$ is equal to the sign of 
\[
\frac{1}{ne^{l\pi\cot\pi/n}}\sum_{k=0}^{n-1}e_{k}^{-(m+n)}\sum_{j=0}^{n-1}\frac{1}{j!}\left(\frac{l\pi e_{k}}{\sin\pi/n}\right)^{j}
\]
if this quantity is nonzero. By the fact that 
\[
\sum_{k}e_{k}^{a}=\begin{cases}
0 & \mbox{ if }n\nmid a\\
n(-1)^{a/n} & \mbox{ if }n|a
\end{cases},
\]
the expression becomes
\[
\frac{1}{e^{l\pi\cot\pi/n}}\frac{(-1)^{p+1}}{r!}.
\]
The lemma follows.

\end{proof}

We now provide an argument that shows the function $h(\theta)$ has
$p$ real roots on the interval $[0,\pi/n)$. Lemmas \ref{alternatesigns},
\ref{smalltheta}, and \ref{largetheta} show that the sign of $h(\theta)$
alternates when $\theta$ varies among $p+1$ values 
\[
\theta=\frac{\pi}{n}-\frac{l\pi}{m}
\]
 where $l=h+\{m/n\}$, $1\le h\le p-1$, and 
\[
\theta=0^{+},\frac{\pi}{n}^{-}.
\]
The claim then follows from the Intermediate Value Theorem.

Each real root $\theta\in[0,\pi/n)$ of $h(\theta)$ yields a certain
number of roots $z$ of $H_{m}(z)$ by Lemma \ref{HthetaLemma} and
(\ref{eq:zqconnection}). If $z$ is such a root, it lies on the curve
$\mathcal{C}$ since the imaginary part of $B^{n}(z)/A(z)$ is zero
and $0\le(-1)^{n}\Re\left(B^{n}(z)/A(z)\right)\le n^{n}/(n-1)^{n-1}$
by (\ref{eq:ChebyshevU}), (\ref{eq:zqtheta}), and (\ref{eq:thetaineq}).
The number of roots of $H_{m}(z)$ on $\mathcal{C}$ is at least the
degree of $H_{m}(z)$ by the argument after the proof of Lemma \ref{zerosB}.

We end this paper with an argument which shows the density of the
roots of $H_{m}(z)$ on $\mathcal{C}$. We first notice that the $p+1$
values of $\theta$ mentioned above are dense on the interval $[0,\pi/n]$
when $m\rightarrow\infty$. From (\ref{eq:zqtheta}), the rational
map $(-1)^{n}B^{n}(z)/A(z)$ maps an open neighborhood of a point
on $\mathcal{C}$ to an open set which contains a solution $\theta$
of $h(\theta)$ where $q=e^{2i\theta}$. Lemma \ref{HthetaLemma}
shows that there is a solution of $H_{m}(z)$ in $U$. The density
of the roots of $H_{m}(z)$ on $\mathcal{C}$ follows.

\end{document}